\documentclass[10pt]{amsart}

\usepackage{amsmath,amscd,amssymb,amsthm, }
\usepackage{cite}
\usepackage{enumerate}
\usepackage{todonotes}
\usepackage{color}

\usepackage[toc,page]{appendix}

\newtheorem{proposition}{Proposition}[section]
\newtheorem{lemma}[proposition]{Lemma}
\newtheorem{theorem}[proposition]{Theorem}

\theoremstyle{definition}
\newtheorem{remark}[proposition]{Remark}
\newtheorem{definition}[proposition]{Definition}

\newcommand{\D}{\Delta}
\renewcommand{\O}{\Omega}
\renewcommand{\l}{\lambda}
\renewcommand{\P}{\mathbb{P}}
\renewcommand{\O}{\mathcal{O}}

\DeclareMathOperator{\Aut}{Aut}

\DeclareMathOperator{\Rees}{Rees}

\DeclareMathOperator{\Sym}{Sym}

\DeclareMathOperator{\Tr}{Tr}

\DeclareMathOperator{\diam}{diam}

\newcommand{\R}{\mathbb{R}}
\newcommand{\A}{\mathbb{A}}

\newcommand{\Q}{\mathbb{Q}}
\newcommand{\G}{\mathbb{G}}

\newcommand{\hH}{\bar{H}}

\newcommand{\scL}{\mathcal{L}}

\newcommand{\scM}{\mathcal{M}}

\newcommand{\scX}{\mathcal{X}}

\newcommand{\X}{\mathcal{X}}

\renewcommand{\L}{\mathcal{L}}
\renewcommand{\G}{\mathbb{G}_m}

\DeclareMathOperator{\df}{df}

\DeclareMathOperator{\chow}{chow}

\title{Tits buildings and K-stability}

\author[G. Codogni]{Giulio Codogni}
\address{EPFL, SB MATHGEOM CAG, MA B3 635 (B\^{a}timent MA), Station 8, CH-1015 Lausanne}
\email{giulio.codogni@epfl.ch}


\begin{document}
\maketitle

\begin{abstract}
A polarized variety is K-stable if, for any test configuration, the Donaldson-Futaki invariant is positive. In this paper, inspired by classical geometric invariant theory, we describe the space of test configurations as a limit of a direct system of Tits buildings. We show that the Donaldson-Futaki invariant, conveniently normalized, is a continuous function on this space. We also introduce a pseudo-metric on the space of test configurations. Recall that K-stability can be enhanced by requiring that the Donaldson-Futaki invariant is positive on any admissible filtration of the co-ordinate ring. We show that admissible filtrations give rise to Cauchy sequences of test configurations with respect to the above mentioned pseudo-metric.
\end{abstract}

\begin{section}{Introduction}

The Yau-Tian-Donaldson conjecture predicts that the existence of a canonical metric on a polarized variety $(X,L)$ is equivalent to an appropriate algebraic notion of stability, which should generalize the classical geometric invariant theory stability. 

In classical geometric invariant theory, the Hilbert-Mumford criterion asserts that a point is stable if and only if the Hilbert-Mumford weight is positive on every non-trivial one parameter subgroup.

The suggested generalization of geometric invariant theory is K-stability, which says that a polarized variety $(X,L)$ is K-stable if for every non almost trivial test configuration the Donaldson-Futaki invariant is positive. In this theory, the role of one parameter subgroups is played by test configurations, the Donaldson-Futaki weight is a Hilbert-Mumford weight, and the Hilbert-Mumford criterion is turned into a definition.

Nowadays, it is widely accepted that the notion of K-stability should be enhanced. In \cite{ICM}, a stronger notion is proposed: test configurations are identified with finitely generated admissible filtrations, and $(X,L)$ is called $\hat{K}$-stable if the Donaldson-Futaki invariant is positive on every admissible filtration, not just on the finitely generated ones. The Donaldson-Futaki invariant of a non-finitely generated admissible filtration is defined by approximating the filtration with honest test configurations, and then taking the limit along this approximation. We will recall the relevant definitions in Section \ref{sec:filtrations}.

In classical geometric invariant theory, non-zero one parameter sup-groups are parametrised by the rational points $\D(\Q)$ of a a space $\Delta$, which is usually called the Tits building or flag complex. The Hilbert-Mumford weight, conveniently normalized, becomes a function on $\Delta$. This space can be endowed with various geometric structures, which can be used for different goals; for example, they are used to show the existence and the uniqueness of a maximally destabilizing one parameter subgroup for unstable points, see \cite{Kempf} and \cite{Rousseau}.

As observed by Y. Odaka in \cite{Odaka}, test configurations are parametrised by an appropriate direct system of Tits buildings $\{\Delta_r(\Q)\}_{r\in \mathbb{N}}$, where $\Delta_r$ is the Tits building parametrising one parameter subgroups of $SL(H^0(X,rL))$. We denote by $\Delta_{\infty}(\Q)$ this direct limit, and we investigate two different structures that one can put on this space.

Tits building can be defined as abstract simplicial complex, this point of view gives a topology on $\Delta_r$ which we call the  \emph{simplicial topology}. The Hilbert-Mumford weight is continuous with respect to this topology. In Theorem \ref{thm:cont0}, we will show that the morphisms appearing in the direct system $\{\Delta_r(\Q)\}_{r\in \mathbb{N}}$ are continuous with respect to the simplicial topology on $\Delta_r(\Q)$. We call simplicial topology the direct limit topology induced on $\Delta_{\infty}(\Q)$. The following result, proved in Section \ref{sec:DF-simplicial}, is a corollary of the above mentioned continuity result.

\begin{theorem}
The normalized Donaldson-Futaki weight is continuous with respect to the simplicial topology on the sub-set $\mathcal{T}$ of $\Delta_{\infty}(\mathbb{Q})$ of non almost trivial test configurations.
\end{theorem}

Let us stress that the maps appearing in the direct system $\{\Delta_r(\Q)\}_{r\in \mathbb{N}}$ do not preserve the simplicial structures, hence $\Delta_{\infty}(\Q)$ does not have a natural simplicial structure.

The second structure that we want to discuss is a metric structure. Each Tits building $\Delta_r$ can be endowed with a metric $d_r$; we call this metric the \emph{Tits metric}, and the induced topology the \emph{Tits topology}. The Tits topology is coarser than the simplicial topology. Using the direct system $\{\Delta_r(\Q)\}_{r\in \mathbb{N}}$, we are able to induce in Definition \ref{def:space_conf} a limit pseudo-metric $d_{\infty}$ on $\Delta(\Q)$. This metric is defined as a limsup, and in Proposition \ref{prop:limsup} we show that this limsup is actually a limit. Our next result shows that this metric gives a convenient set-up to study $\hat{K}$-stability.

\begin{theorem}[=Theorem \ref{thm:cauchy}]\label{thm:intro}
Let $F$ be non-finitely generated admissible filtration with non-zero $L^2$ norm; then the sequence of points in $\Delta_{\infty}$ associated to the sequence of test configurations approximating $F$ is a Cauchy sequence for the pseudo-metric $d_{\infty}$.
\end{theorem}
The notions of admissible filtrations and $L^2$ norm will be recalled later on.

In Section \ref{sec:analogy}, we explain the relation between classical Tits building and symmetric spaces. We suggest a relation between the Tits building $\Delta_{\infty}$ and the space of K\"{a}hler metrics. Taking this point of view, it is natural to ask about maximal flat subspaces of the space of K\"{a}hler metrics.

The interplay between the simplicial and the Tits topology, as well as the behaviour of the Donaldson-Futaki invariant with respect to the Tits metric, are topics which deserve further investigations. Mimicking the arguments used in geometric invariant theory by \cite{Kempf} and \cite{Rousseau}, a convenient convexity result about the Donaldson-Futaki invariant would imply the existence and unicity of a maximally destabilizing test configurations.

\begin{subsection}*{Relations with other works}

It is possible to define a map from the space $\Delta_{\infty}$ to an appropriate quotient of the space of non-archimedean metrics on the analytification of $(X,L)$ introduced in \cite{BHJ2} and \cite{prep}. This map should be continuous for the simplicial topology. We do not investigate this topic in this note.

K-stability can also be defined in the non-projective setting, see \cite{DervanRoss}, \cite{Der2} and \cite{Zak}. In this set-up, a test configuration is a space $\scX$ endowed with a K\"{a}hler form rather than a line bundle. This configurations do not come naturally from the action of a one parameter sub-group, so Tits building are not available in this setting. It would be interesting to find an alternative way to describe the space $\Delta_{\infty}$.

The automorphism group $\Aut(X,L)$ acts naturally on $\Delta_{\infty}$ preserving the pseudo-metric. When $\Aut(X,L)$ is not reductive, the pair $(X,L)$ is expected to be not K-stable, or at least not $\hat{K}$-stable. In \cite{me}, it is introduced a canonical admissible filtration, called Loewy filtration, which should be destabilising exactly when $\Aut(X,L)$ is not reductive. An interpretation of the Loewy filtration as a Cauchy sequence in $\Delta_{\infty}$ could be a useful step towards the proof of this conjecture.

\end{subsection}

\begin{subsection}*{Notations}
We work over an algebraically close field $\Bbbk$ of characteristic zero. We fix a normal projective variety $X$ of dimension $n$ and a very ample and projectively normal line bundle $L$ over $X$. We use the additive notation for line bundles, so $mL=L^{\otimes L}$.
\end{subsection}
\begin{subsection}*{Acknowledgements} We had the pleasure and the benefit of conversations about the topics of this paper with S. Boucksom, R. Dervan, M. Jonsson, J. Ross, J. Stoppa and F. Viviani. The author was also supported by the FIRB 2012 -Moduli spaces and their applications, and the ERC StG 307119 - Stability in Algebraic and Differential Geometry.
\end{subsection}
\end{section}

\begin{section}{Tits buildings}\label{sec:prel}
In this section, following \cite{Serre} and \cite[Section 2.2]{GIT}, we recall the definition of the \emph{Tits building} $\D$ associated to a finite dimensional complex vector space $V$, and some of its properties. In the literature, Tits buildings are sometime called \emph{spherical buildings} or \emph{flag complexes}.

Let $m$ be the dimension of $V$, and assume that $m\geq 3$. The first definition of $\D$ is as an abstract simplicial complex. Simplexes correspond to parabolic sub-group of $SL(V)$; a simplex corresponding to a parabolic group $P_1$ lies in the boundary of a simplex corresponding to a parabolic group $P_2$ if and only $P_2\subset P_1$. Vertexes are given by maximal parabolic subgroups; maximal simplexes are $m-2$ dimensional. %In particular, maximal simplexes correspond to minimal parabolic groups.

Recall that parabolic subgroups correspond to flags of $V$: to a flag we associated its stabiliser. We thus have the following equivalent description of $\D$: each vertex corresponds to a proper vector subspace of $V$; a group of vertexes form a simplex if and only if the associated subspaces form a flag in $V$. 

We can now start enhancing the structure of $\D$.  We identify each simplex with the standard one, in particular we have co-ordinates $x_i$; let $\D(\Q)$ be the set of point with rational co-ordinates. We introduce the following definition
\begin{definition}[Weighted flag]\label{def:weighted_flag}
A weighted flag is the data of a flag 
$$
\{0\} \subset F_1V \subset F_2V \cdots F_{k-1} V \subset F_kV=V
$$
and weights $w=(w_1, \dots , w_k)$ such that $w_i< w_{i+1}$ and $\sum w_i=0$. The weights can be either rational or real numbers. Two flags $(F,w)$ and $(G,w')$ are equivalent if there exists a constant $c$, called the scaling constant, such that $F_iV=G_iV$ and and $w_i=cw_i'$ for every $i$.

The weight of a vector $v$ is the maximum $w_i$ such that $v\in F_i V$.

We say that a basis $\{v_1, \dots, v_n\}$ of $V$ is \emph{adapted} to a flag $F$ if, for every $i$, there exists a subset of $\{v_1, \dots, v_n\}$ which forms a basis of $F_iV$.
\end{definition}

The Tits building $\D$ parametrise weighted flag up to equivalence: the flag corresponds to the simplex, and the weight to the co-ordinates of the point. 

We now associate to each one parameter subgroup $\l$ of $SL(V)$ a weighted flag, hence a point $[\l]$ of $\D(\Q)$. Let $w_1,\dots, w_k$ be the weights of $\lambda$, ordered in an increasing way; let
$$
F_i(V)=\bigoplus_{j\leq i} V_{w_j}
$$
where $V_{w_j}$ is the eigenspace of weight $w_j$ of $\lambda$. Assigning weight $w_i$ to $F_iV$, we obtain the flag associated to $\l$. As shown in \cite[Proposition 2.6]{GIT}, the parabolic sub-group $P(\lambda)$ stabilising this flag consists of all $g$ in $SL(V)$ such that the limit $\lim_{t\to 0} \lambda(t)g\lambda(t)^{-1}$ does exist. This limit, when it exists, centralizes $\lambda$, so it preserves the eigenspaces of $\lambda$; see the proof of \cite[Proposition 2.6]{GIT} for a more precise description of the limit.

Two 1PS's gives the same point in $\D$ if and only if the associated flags are equivalent. In other words, $\D(\Q)$ is equal to set of one parameter subgroups of $SL(V)$ modulo the equivalence relations:
\begin{align*}
\lambda \sim \gamma \quad  \textrm{if}  \quad \lambda=p\gamma p^{-1} \quad p\in P(\lambda) \\
\lambda \sim \gamma \quad   \textrm{if} \quad \lambda^a=\gamma^b \quad a,b \in \mathbb{Z}
\end{align*}

The next piece of structure is given by the apartment. Apartments correspond to maximal tori of $SL(V)$: given a maximal torus $T$, the corresponding apparent $A_T$ is the closure in $\D$ of the one parameter sub-groups of $T$. A flag $F$ is in $A_T$ if and only if the eigenvectors $v_1,\dots , v_n$ of $T$ form a basis adapted to $F$. The key remark is that an apartment is a finite simplicial complex homeomorphic to a sphere, or a simplicial sphere for short. The following standard lemma will be very important
\begin{lemma}(\cite[Lemma II.2.9]{GIT})
Given two points $p$ and $q$ of $\D$, there exists at least an apartment containing both of them.
\end{lemma}
The previous lemma can also be interpreted in the following way: given two points $p$ and $q$ in $\Delta(\Q)$, there exists two commuting 1PS's $\lambda$ and $\gamma$ of $SL(V)$ such that $p=[\lambda]$ and $q=[\gamma]$. %We will also say that the apartment is adapted to $p$ and $q$.

The building, so far, is an abstract simplicial complex. Looking at its geometric realisation, we can endow it with a topology, which we call the \emph{simplicial topology}.  The simplicial complex $\Delta$ is not locally of finite type, so we need some care in the description of this structure. Apartments are finite simplicial complex homeomorphic to a sphere. On the entire space $\Delta$, the topology is defined as the direct limit of the topology of finite sub-complexes. Since any finite sub-complex is contained in a finite number of apartments, a subset $U$ of $\Delta$ is open if and only if its intersection with any apartment is open. Similarly, a function on $\Delta$ is continuous with respect the simplicial topology if and only if its restriction to each apartment is continuous.

We are now in position to introduce the Tits metric on $\Delta(\Q)$.
\begin{definition}[Tits metric]\label{def:TitsMetric}
Let $p$ and $q$ be two points of $\Delta(\Q)$; pick two commuting 1PS's $\lambda$ and $\gamma$ such that $p=[\lambda]$ and $q=[\gamma]$, and write $\lambda=\exp tA$ and $\gamma=\exp tB$; then we let
$$
d(p,q)=\arccos\left(\frac{\Tr(AB)}{\sqrt{\Tr(A^2)\Tr(B^2)}}\right)
$$
\end{definition}
One can show that this definition is independent of the chosen one parameter subgroups, see for instance \cite[Section 2.2]{GIT}. Moreover, this metric can be extended by continuity to $\Delta$.

Let us describe an interpretation of the Tits metric as angular distance. Take a maximal torus $T$ containing both $\lambda$ and $\gamma$, this gives an apartment $A_T$ containing both $p$ and $q$.
 Let $\Gamma(T)$ be the lattice of one parameter subgroups of $T$. The Killing metric on $\Gamma(T)$ is a quadratic form which is equivariant for the action by conjugation of the normalizer of $T$ in $G$, it is unique up to a scalar. Denote by $E$ the space $\Gamma(T)\otimes \mathbb{R}$ equipped with the Killing metric. Then, $A_T$ can be identified with the unit sphere in $E$, and the Tits metric is nothing but the angular distance. Since the Killing metric unique up to a scalar, the angular distance on $A_T$ is uniquely defined.

Since any two points are contained in an apartment, and the apartment is isometric to a sphere endowed with the angular distance, we have that any two points can be connected by a geodesic and $\diam(\Delta)=\pi$. The geodesic is not unique because, for instance, two points can be contained in many different apartments, and the geodesic constructed above depends on the apartment.

The topology induced by the Tits metric on each apartment is equal to the simplicial topology. However, on $\Delta$, the topology induced by the Tits metric is coarser than the simplicial topology.

\end{section}

\begin{section}{Tits building and test configurations}\label{sec:Tits_and_test}
Let $X$ be a projective variety over an algebraically closed field of characteristic zero, and $L$ a very ample and projectively normal line bundle on $X$. We also fix a generator $t$ of the space of one parameter subgroups of $\G$, and faithful action of $\G$ on $\A^1$; let $0$ be the fixed point of the action and $1$ another point of $\A^1$. We recall the definition of test configuration, which is due to S. Donaldson \cite[Definition 2.1.1]{Don2}.
\begin{definition}[Test configuration]\label{def:tc}
Let $r$ be a positive integer. An exponent $r$ test configuration $(\scX,\scL)$ for $(X,L)$ consist of the following data
\begin{enumerate}
\item a scheme $\scX$ together with a flat map $\pi \colon \scX \to \A^1$;
\item a $\G$ action on $\scX$ such that the morphism $\pi$ is equivariant;
\item a relatively ample line bundle $\scL$ on $\scX$ together with a linearisation of the $\G$ action.
\end{enumerate}
Moreover, we require that the fibre over $1$ is isomorphic to $(X,rL)$.

A test configuration is very (respectively semi-) ample if $\scL$ is very (semi-) ample. A test configuration is trivial if $(\scX,\scL)$ is isomorphic to $(X\times \A^1, rL\boxtimes \mathcal{O}_{\A^1})$, and the $\G$ action is trivial on $X$. A test configuration is normal if $\scX$ is normal. 

Let $\nu\colon \hat{\scX}\to\scX$ be the normalization, then $(\hat{\scX},\nu^*\scL)$ has a natural structure of test configuration, we call it the normalization of $(\scX,\scL)$. A test configuration is almost trivial if its normalization is trivial.

A non-polarized test configuration is the datum of an $\scX$ with a $\G$ action as above, without the choice of a line bundle $\scL$.
\end{definition}

Basic properties of test configurations are described in \cite[Section 2]{BHJ}. There are three main operation one can perform on test configurations.
\begin{definition}\label{def:action}

\begin{description}
\item[Base change] Let $b_p\colon \A^1 \to \A^1$ be the map defined by $z\mapsto z^p$. We can make a base change $(\scX,\scL)$ via $b_p$ obtaining a new test configuration.
\item[Scaling] Consider the trivial action of $\G$ on $X$, and fix a faithful lifting of this action to $\scL$, so that the induced action on $H^0(\scX_0,\scL_0)$ is a homothety. We can scale the action of $\G$ on $\scL$ by adding $c$ times this action, where $c$ is in $\mathbb{Z}$.
\item[Rising the line bundle] We can replace $\scL$ with $m\scL$, for any positive integer $m$.
\end{description}
\end{definition}

We now recall the definition of the $L^2$ norm of a test configuration. For every $k$, the test configuration gives rise to a $\G$ action on $H^0(\scX_0,k\scL_0)$; let $T_k$ be an infinitesimal generator of this action. We denote by $\underline{T}_k$ the traceless part of $T_k$, in symbols
$$\underline{T}_k=T_k-\frac{\Tr(T_k)}{h^0(\scX_0,k\scL_0)}Id$$
Then $\Tr(\underline{T}_k^2)$ is, for $k$ big enough, a degree $n+2$ polynomial in $k$, where $n$ is the dimension of $X$, see for instance \cite[Equation 4]{Gabor} or \cite[Theorem3.1]{BHJ}. We let
$$
||(\scX,\scL)||_{L^2}^2=\lim_{k\to \infty}(kr)^{-n-2}\Tr(\underline{T}_k^2)
$$
Remark that $||(\scX,\scL)||_{L^2}=||(\scX,m\scL)||_{L^2}$ for every $m$.

Let now $V_r=H^0(X,rL)^{\vee}$. Given a 1PS $\lambda$ of $SL(V_r)$, we can construct a test configuration by taking the flat closure of the $\lambda$-orbit of $X$ in $\mathbb{P} V_r$. Any very ample exponent $r$ test configuration arise as orbit of a 1PS of $GL(V_r)$, see \cite[Proposition 3.7]{RossThomas}. By performing a base change and scaling the linearisation, we can always assume that this 1PS lies in $SL(V_r)$. We have now the following key observation of Odaka \cite{Odaka}.
\begin{theorem}
Let $\Delta_r$ be the  Tits building of $V_r$. Then, points of $\Delta_r(\Q)$ are in bijective correspondence with very ample exponent $r$ test configurations, modulo base change and scaling.
\end{theorem}
\begin{proof}
The only thing we have to check is that if the weighted flags associated to two 1PS's are equivalent, then also the corresponding test configuration are equivalent. This is done in \cite[Theorem 2.3]{Odaka}.
\end{proof}

Almost triviality of a test configuration can be characterized in term of the associated filtration of $V_r$, see \cite[Proposition 2.12]{BHJ}. The following lemma, which is contained in the proof of \cite[Proposition 3.7]{RossThomas} is also very important

\begin{lemma}
There exists a $\G$-equivariant trivialization of $\pi_*\scL$; this gives a $\G$-equivariant isomorphism between $H^0(X,rL)^{\vee}$ and $H^0(\scX_0,\scL_0)$, where the $\G$ action on the first vector space is given by the one parameter subgroup inducing the test configuration.
\end{lemma}

From the point of view of K-stability, it is quite natural to identify the test configuration $(\scX,\scL)$ and $(\scX,m\scL)$. Because of this, we look at the direct systems formed by the buildings $\Delta_r$ and the morphisms
\begin{displaymath}
\begin{array}{ccccc}
\iota_{r,k}\colon & \Delta_r(\Q) &\to & \Delta_{rk}(\Q) \\
&(\scX,\scL) & \mapsto & (\scX,k\scL)
\end{array}
\end{displaymath}

\begin{theorem}[= Theorem \ref{thm:cont}]\label{thm:cont0}
For every $i$ and $k$, the map $\iota_{r,k}$ is continuous for the simplicial topology.
\end{theorem}
We postpone it to Section \ref{sec:map}; let us point out that in the proof we also describe explicitly the morphisms $\iota_{r,k}$, and these morphisms do not preserve the simplicial structure of $\Delta_r$. In other words, if one sees $\Delta_r$ as a direct limit of simplicial spheres, the maps $\iota_{r,k}$ are well defined on the resulting topological space $\Delta_r$, but do not preserve the direct system structure of $\Delta_r$, and it does not make sense to ask if the two limit commute. We are now going to define the central object of study of this paper.

\begin{definition}[Space of test configurations]\label{def:space_conf} The space of test configurations is the space $\Delta_{\infty}(\Q)$ defined as the direct limit
$$\Delta_{\infty}(\Q):=\lim_r\Delta_r(\Q)\,.$$
The simplicial topology on $\Delta_{\infty}(\Q)$ is the direct limit of the simplicial topology on $\Delta_r(\Q)$.

The pseudo-metric on $d_{\infty}$ is the pseudo-metric given by
$$
d_{\infty}(p,q)=\limsup_r d_r(p,q)
$$
where $d_r$ is the Tits metric on $\Delta_r(\Q)$. The Tits topology on $\Delta_{\infty}(\Q)$ is the topology induced by $d_{\infty}$.
\end{definition}

Remark that $p$ and $q$ can be seen as points of $\Delta_r$ for every $r$ divisible enough, so the previous expression for $d_{\infty}$ makes sense. Moreover, since $\diam(\Delta_r)=\pi$ for every $r$, $d_{\infty}$ is finite and $\diam(\Delta_{\infty})\leq \pi$. 

The space $\Delta_{\infty}(\Q)$ parametrizes all test configurations, modulo the three operations defined introduced in \ref{def:action}, namely modulo scaling, base change and rising the line bundle.

\begin{proposition}\label{prop:limsup}
The limsup appearing in Definition \ref{def:space_conf} is actually a limit; in other words,
$$
d_{\infty}(p,q)=\lim_r d_r(p,q)
$$
\end{proposition}
\begin{proof}
Let $(\X_1,\L_1)$ and $(\X_2,\L_2)$ be very ample test configurations associated to $p$ and $q$. By raising $\L_1$ and $\L_2$ to suitable powers, we can assume that they have the same exponent. When $r$ is divisible by the exponent, we have the Tits metric
$$
d_r(p,q)=\arccos\left(\frac{\Tr(A_rB_r)}{\sqrt{\Tr(A_r^2)\Tr(B_r^2)}}\right)\,,
$$
where $A_r$ and $B_r$ are generators of two commuting one parameter subgroups of $SL(H^0(X,rL)^{\vee})$ inducing respectively $(\X_1,r\L_1)$ and $(\X_2,r\L_2)$. The denominator of $d_r(p,q)$ is well known to be, for $r$ divisible enough, a polynomial of degree $n+2$, see for instance \cite[Equation 4]{Gabor} or \cite[Theorem 3.1]{BHJ}. We are going to show that also the numerator is a polynomial of degree $n+2$.

Choose a non-polarised test configuration $\X$ dominating equivariantly both $\X_1$ and $\X_2$; this can be constructed by resolving simultaneously the indeterminacy of the maps $X\times \mathbb{P}^1\dashrightarrow \X_i$, cf \cite[Section 6.6]{BHJ}. Let $\alpha$ be the $\mathbb{G}_m$ action on $\X$. Denote by $\scM_i$ be the pull-back of $\L_i$ to $\X$. The restriction of $\scM_1+ \scM_2$ on $\X_t$, for $t\neq 0$, is isomorphic to $2rL$, hence $H^0(\X_0,\scM_1+ \scM_2|_{\X_0})^{\vee}$ can be identified in a $\G$ equivariant way with  $H^0(X,2rL)^{\vee}$ and the infinitesimal generator of the action of $\alpha$ is exactly $A_r+B_r$. By applying \cite[Theorem 3.1]{BHJ} we show that $\Tr(A_r+B_r)^2$ is a polynomial of degree $n+2$. Since also $\Tr A_r^2$ and $\Tr B_r^2$ are polynomial of degree $n+2$, we conclude that the same is true for $\Tr A_kB_k$.

\end{proof}

It is also natural to consider the space
$$
\Delta_{\infty}:= \lim_r\Delta_r
$$
endowed with its simplicial topology. We have a natural inclusion $\Delta_{\infty}(\Q)\subset \Delta_{\infty}$, and we can extend $d_{\infty}$ to a metric on $\Delta_{\infty}$. We do not know about relation between $\Delta_{\infty}$ and the completion of $\Delta_{\infty}(\Q)$ with respect to $d_{\infty}$.

\end{section}

\begin{section}{Donaldson-Futaki invariant and the symlicial topology}\label{sec:DF-simplicial}
We first briefly recall some facts about the Hilbert-Mumford weight, following \cite[Chapter 2]{GIT}. Let the group $SL(V)$ act on a projective variety $Z$, and linearise the action to a line bundle $H$. Pick a closed point $z$ in $Z$. For any 1PS $\lambda$ of $SL(V)$ we can consider the Hilbert-Mumford weight $\mu(\lambda)$ with respect to $z$ and $H$.

Fix now an $SL(V)$ invariant norm $|| \, - \,||$ on the 1PS's of $SL(V)$. The ratio $\nu(\lambda)=\mu(\lambda)/||\lambda||$ is a well defined function on the Tits building $\Delta(V)$; moreover, $\nu$ is continuous for the simplicial topology.

Following \cite[Section 3]{RossThomas} and \cite{Gabor}, we introduce the Chow weights and the Donaldson-Futaki weight. Fix an exponent $r$, and let $V_r=H^0(X,rL)$. We choose as $SL(V_r)$-invariant norm on the space of 1PS's of $SL(V_r)$ the norm $||\exp(tA)||=r^{-n-2}\Tr_{V_r}A^2$. Remark that $\lambda$ is already taken in the special linear group, so $A$ is traceless. In particular, $||\exp(tA)||$ is equal to $r^{-n-2}\Tr(\underline{T}_1^2)$, where $\underline{T}_1$ is the operator introduced in Section \ref{sec:Tits_and_test}.

The group $SL(V_r)$ acts on the appropriate Hilbert scheme $Z_r$, and the variety $X$ gives a point $[X]$ in $Z_r$. Choosing the correct line bundle on $Z_r$, the associated normalized Hilbert-Mumford weight is the normalised Chow weight:
$$
\chow_r\colon \Delta_r \to \R
$$
This line bundle is the pull-back of the Chow line bundle from the Chow scheme, via the cycle-class map from the Hilbert scheme to the Chow scheme. The normalization of he Chow line bundle is such that the $r$-th normalized Chow weight of an exponent $r$ test configuration is 
$$
\chow_r(\X,\L)=||\lambda||^{-1}\frac{ra_0}{b_0}
$$
where $h(k)=a_0k^n+O(k^{n-1})$ is the Hilbert polynomial of $(X,rL)$, and $w(k)=b_0k^{n+1}+O(k^n)$ is the trace of the operator $T_k$ introduced in Section \ref{sec:Tits_and_test}. Remark that $w(r)=0$, because we started off with a $\lambda$ in the special linear group, however $w(k)$ is a non-trivial polynomial of degree $n+1$ for $k$ big enough.

Pulling-back via the maps $\iota_{r,k}$, we have the higher Chow weights
$$
\chow_{rk}\colon \Delta_r \to \R \, ,
$$
so that
$$
\chow_{kr}(\X,\L)=||\lambda||^{-1}\left(\frac{kra_0}{b_0}-\frac{w(k)}{h(kr)}\right)
$$

Let $\mathcal{T}\subset \Delta_{\infty}(\Q)$ be the subset of test-configurations with non-zero $L^2$ norm, and $\mathcal{T}_r$ its intersection with $\Delta_r$. Fix a point in $\mathcal{T}_r$, the value of the Chow weight at that point is, for $k$ big enough, equal to a Laurent polynomial
$$
\chow_{kr}=\df+\ell(k)
$$
where $\df$ is the constant term and $\ell(k)$ is the principal part of the Laurent polynomial, in particular $\ell(k)$ converges to zero when $k$ goes to infinity. This gives an invariant
$$
\df\colon \mathcal{T} \to \R
$$
defined as $\df(p)=\lim_k \chow_{kr}(p)$, where $r$ is such that $p$ lies in $\Delta_r$.

The invariant $\df$ is by definition the Donaldson-Futaki invariant of a test configuration divided by its $L^2$ norm.

\begin{lemma}\label{lem:uni}
Fix $r$, then there exists a positive integer $K$ such that $\chow_{kr}$ is a Laurent polynomial for all exponent $r$ test configuration and all $k$ divisible by $K$.
\end{lemma}
\begin{proof}
Fixed a test configuration $(\X,\L)$, the Chow invariant $\chow_{kr}$ is a polynomial as soon as $H^i(\X_0,k\L_0)$ vanishes for all $i>0$, see \cite[Theorem 3.1 and Corollary 3.2]{BHJ}. Fixed the exponent, central fibres are parametrized by a Hilbert scheme, so the result follows from a general statement of the form: if $T$ is a noetherian scheme, and $Y\to T$ is a projective morphism with a relatively ample line bundle $L$, then there exists a $K$ such that $H^i(Y_t,kL_t)=0$ for all $t$ in $T$, all $i>0$, and all $k$ divisible by $K$; this is well-known, see for instance\cite[Theorem 1.2.13 and its proof]{Laz}.
\end{proof}

We have now the following proposition.
\begin{proposition}
The normalised Donaldson-Futaki invariant $\df$ is continuous with respect to the simplicial topology on $\mathcal{T}\subset \Delta_{\infty}(\Q)$.
\end{proposition}
\begin{proof}
Since the topology on $\Delta_{\infty}(\Q)$ is the direct limit topology, it is enough to show that $\df$ is continuous when restricted to $\mathcal{T}_r$, for every $r$. We know that $\chow_{kr}$ is continuous on $\mathcal{T}_{kr}$ for every $k$ and $r$; since, by Theorem \ref{thm:cont0}, the maps $\iota_{k,r}$ are continuous, $\chow_{kr}$ is continuous on $\mathcal{T}_r$. By Lemma \ref{lem:uni}, for $k$ divisible enough $\chow_{kr}$ is a Laurent polynomial in $k$, so all its coefficients have to be continuous as functions on  $\mathcal{T}_r$. This proves the claim.
\end{proof}
\begin{remark} At least for smooth varieties over the complex numbers, because of Donaldson work \cite{Don}, we know that $\df$ is bounded below on $\mathcal{T}$. The lower bound can be described in term of the curvature of K\"{a}hler metrics in the class $c_1(L)$.
\end{remark}

When $X$ is a normal variety, a test configuration has zero $L^2$ norm if and only if it is almost trivial, \cite[Theorem 1.3]{Ruadhai_Uniform} and \cite[Corollary B]{BHJ}. Let us now give the definition of K-stability
\begin{definition}[K-stability] A normal polarised variety $(X,L)$ is K-semistable if $\df(\scX,\scL)\geq 0$ for every test configuration $(\scX,\scL)$. It is K-stable if it is K-semistable and $\df(\scX,\scL)=0$ if and only if $(\scX,\scL)$ is almost trivial.
\end{definition}
\end{section}
\begin{section}{Filtrations and the completion with respect to the Tits metric}\label{sec:filtrations}
In this section we study filtrations of the co-ordinate ring $R$ of $(X,L)$. Recall that the K-stability of $(X,L)$ is equivalent to the K-stability of $(X,kL)$ for every $k$, so we can assume without loss of generality that $L$ is projectively normal.
\begin{definition}[Admissible filtration]\label{filtration}
A filtration $F$ of $R$ is the datum of an increasing flag on each graded piece $H^0(X,kL)=V_k$, indexed by $\mathbb{Z}$. We say that the filtration is
\begin{description}
\item[Multiplicative] if $$F_iV_a\otimes F_jV_b\to F_{i+j}V_{a+b}\,,$$ for every $a$, $b$, $j$ and $k$;
\item[Point wise right bounded] If for every fixed $k$ we have that $$F_iV_k=V_k$$ for $i>>0$; this is also said \emph{exhaustive};
\item[Linearly left bounded] If there exists a negative constant $C$ such that $$F_{Ck}V_k=\{0\}$$ for every $k$.
\end{description}
A filtration is \emph{admissible} if it satisfies the three properties listed above. We let $F_iR=\oplus_k F_iH^0(X,kL)$.
\end{definition}
There are two operations that we can perform on filtrations. We can scale them, which means replacing $F_i$ with $F_{ci}$ for some fixed constant $c$, and we can shift them, which means replacing $F_iH^0(X,kL)$ with $F_{i+ck}F_iH^0(X,kL)$, for a fixed constant $c$.

Given a multiplicative filtration $F$, we can construct its Rees algebra
$$
\Rees(F)=\bigoplus_i F_i R t^i
$$
we say that a filtration is \emph{finitely generated} if its Rees algebra is finitely generated.

As explained in \cite{witt} and \cite{Gabor}, taking the Proj of the Rees algebra of an admissible finitely generated filtration one obtains a test configuration. More generally, there is a correspondence between finitely generated admissible filtrations of the rings $R$ and test configurations, see \cite[Proposition 2.15]{BHJ}. Under this correspondence, scale the filtration corresponds to a base change, shift the filtration corresponds to scale the linearisation, see Definition \ref{def:action}.

Following \cite{Gabor}, we can approximate a non-finitely generated admissible filtration with finitely generated ones. Let $F$ be an admissible filtration; denote by $\chi^{(m)}$ the $\Bbbk[t]$-sub-algebra of $\Rees(F)$ generated by the finite dimensional vector space $\oplus_i F_i H^0(X,rmL)  t^i\oplus R_it^N$, for $N$ big enough. We now let 
$$F^{(m)}_iH^0(X,mkL)=\{s\in H^0(X,mkL) \textrm{  s.t.  } s t^i \in \chi^{(m)}\}$$ 
this defines a finitely generated admissible filtration of $R$. Let $(\scX^{(m)},\scL^{(m)})$ be the corresponding test configuration. Then one defines
$$
||F||_{L^2}=\liminf_{m\to \infty} ||(\scX^{(m)},\scL^{(m)})||_{L^2}
$$
In \cite{Gabor} it is shown that this liminf is actually a limit.

\medskip

Given a flag $F$, for every $m$ we can construct a weighted flag of $H^0(X,mL)$ up to scaling, in the sense of Definition \ref{def:weighted_flag}. This is done by first giving weight $i$ to the piece $F_i$, and then subtracting a common rational constant to all weights to normalise the trace. The filtration obtained in this way has rational weights; we denote by $p_m$ the corresponding point in $\Delta_m(\Q)$. As explained in \cite[Section 3.2]{Gabor}, the test configuration associated to this weighted filtration is equivalent to the Proj of the Rees algebra of $F^{(m)}$. 

If the filtration is finitely generated, then $F^{(m)}=F$ for $m$ big enough, and the sequence $\{p_m\}$ is eventually constant as a sequence in $\Delta_{\infty}(\Q)$. On the other hand, when the filtration is not finitely generated, the test configurations associated to the points $p_m$ are different, so a non finitely generated filtration defines a non-constant sequence in $\Delta_{\infty}$.

\begin{theorem}\label{thm:cauchy}
Let $p_m$ be the sequence of points in $\Delta_{\infty}(\Q)$ associated to an admissible filtration $F$ such that $||F||_{L^2}\neq 0$; then, this is a Cauchy sequence for the pseudo-metric $d_{\infty}$.
\end{theorem}
\begin{proof}
We need to show that, for every $j$, the distance $d_{\infty}(p_m,p_{jm})$ converges to zero when $m$ goes to infinity. More explicitly, we have to show that
$$
\lim_m \limsup_k \frac{\Tr(A_k^{(m)}A_k^{(jm)})}{\sqrt{\Tr((A_k^{(m)})^2)\Tr((A_k^{(jm)})^2)}}=1
$$
where, for each $m$, the limit is taken on all $k$ divisible by both $m$ and $jm$; the $A_k^{(m)}$ and $A_k^{(jm)}$ are infinitesimal generator of commuting 1PS representing $p_m$ and $p_{jm}$ in $\Delta_k$. Because of the hypothesis on the norm, the limit of the denominator normalised by $k^{-n-2}$ is not zero, so we can compute the limit of the numerator and the denominator separately.

To start with, let us recall that $\limsup_k \sqrt{k^{-n-2}\Tr((A_k^{(m)})^2)}$ converges to the $L^2$ norm of the test configuration associated to $p_m$, and the limit $\lim_m ||p_m||_{L^2}$ is equal to the $L^2$ norm $||F||_{L^2}$ of the filtration $F$, as explained in \cite{Gabor}, see in particular Lemma 8. The same is true for $p_{jm}$.

We now have to deal with the numerator. Fix $m$ and $k$. The multiplicativity of $F$ implies, for every $j$, the following inclusion relation
$$
\chi^{(m)}\cap R^{(jm)}[t]\subseteq \chi^{(jm)}\,,
$$
where $R^{(jm)}$ is the Veronese ring $\oplus_{\ell}H^0(X,jm\ell L)$. This inclusion in turn implies that, for every $i$ and $k$, we have
$$
F_i^{(jm)}H^0(X,mkjL)\subseteq F_i^{(m)}H^0(X,mkjL)\,.
$$
Choosing a basis of $H^0(X,mkjL)$ adopted to both $F^{(m)}$ and $F^{(jm)}$, we can translate the above inclusions in the following inequalities.
$$
\Tr((A_k^{(m)})^2)\leq \Tr(A_k^{(m)}A_k^{(jm)})\leq \Tr((A_k^{(jm)})^2)
$$

Taking the limit on $k$ and then $m$, arguing as before we conclude that
$$\lim_m \limsup_k k^{-n-2}\Tr(A_k^{(m)}A_k^{(jm)})=||F||_{L^2}^2$$
\end{proof}

\end{section}

\begin{section}{Description of the morphisms between Tits buildings}\label{sec:map}
In this section we describes explicitly the maps
\begin{displaymath}
\begin{array}{ccccc}
\iota_{r,k}\colon & \Delta_r(\Q) &\to & \Delta_{rk}(\Q) \\
&(\scX,\scL) & \mapsto & (\scX,\scL^{\otimes k})
\end{array}
\end{displaymath}
Our main result is the following.

\begin{theorem}\label{thm:cont}
For every $i$ and $k$, the map $\iota_{r,k}$ is continuous for the simplicial topology.
\end{theorem}

We can assume without loss of generality that $r=1$; moreover, we fix $k$, so, to simplify the notation, we write $\iota$ for $\iota_{k,r}$, $\Delta$ for $\Delta_r$ and $\Delta_k$ for $\Delta_{rk}$. Let $\Delta_S$ be the Tits building of the vector space $\Sym^k H^0(X,L)$. 

In Section \ref{sec:segre}, we will define a Segre map $S\colon \Delta(\Q)\to \Delta_S(\Q)$, and prove that it is continuous. In Section \ref{sec:retraction}, we will define a retraction map $\rho\colon \Delta_S(\Q)\to \Delta_k(\Q)$, and prove that it is continuous. In Section \ref{sec:prov}, we will show that $\iota$ is actually the composition of $S$ and $\rho$, concluding the proof of Theorem \ref{thm:cont}.

\begin{subsection}{Segre morphism of building}\label{sec:segre}
For an algebraic group $G$, let $\Gamma(G)$ be the set of 1PS's of $G$. Let $V=H^0(X,L)$ and $V_S=\Sym^k H^0(X,L)$. We have a Segre map
\begin{displaymath}
\begin{array}{ccccc}
S\colon &\Gamma(SL(V)) &\to & \Gamma(SL(V_S)) \\
&\gamma & \mapsto & \gamma^{\otimes k}
\end{array}
\end{displaymath}

The Segre map defined on 1PS's induces a morphism of building; we describe directly this morphisms on weighted flags. Denote by $M$ be the collection of multi-indexes $I=(i_1,\dots , i_m)$ with $\sum i_j=k$. Let $\underline{v}$ be a basis of $V$ adapted to a weighted flag $(F,w)$ associated to $\gamma$. We denote by $S(\underline{v})$ the basis of $V_S$ formed by monomials in the element of $\underline{v}$; in particular, for $I\in M$ and $v\in \underline{v}$, we denote by $v^I$ the corresponding monomial in $S(\underline{v})$. Let $T(w)=\sum_{I\in M} w^I$. To each monomial $v^I$ we assign weight $w^I-T(w)$: this defines a weighted flag $(S(F),S(w))$. The weighed flag associated to $S(\gamma)$ is exactly $(S(F),S(w))$, so this gives a description of the map
$$
S\colon \Delta \to \Delta_S
$$
The Segre map preserves apartments in the following sense. Let $A$ be an apartment of $\D$ associated to a basis $\underline{v}$. Then, $S(A)$ is contained in the apartment $A_S$ associated to the basis $S(\underline{v})$.

Let us show that $S\colon A \to A_S$ is continuous for every apartment $A$. Co-ordinates of points in an apartment are given just by the weights. In particular, for $I\in M$, the $I$-th co-ordinate of $(S(F),S(w))$ is $w^I-T(w)$; benign the new co-ordinate a polynomial in the old one, the map $S$ is continuous. Since the simplicial topology on $\Delta$ is the direct limit of the topology on the apartments, we conclude that $S$ is continuous on $\Delta$.

\end{subsection}
\begin{subsection}{Retraction of buildings}\label{sec:retraction}
Let $i\colon W \hookrightarrow V$ be an inclusion of vector spaces. We can define the corresponding retraction of building 
$$
\rho \colon \Delta(V) \to \Delta(W)
$$
as follows. Let $(F,w)$ be a weighted flag in $\Delta(V)$. Choose a basis of $V$ adapted both to $F$ and $W$. This amounts to choose a representative $\gamma$ of $F$ which preserves globally $W$; we denote by $U$ the $\gamma$-invariant complement of $W$ in $V$. We now let $\rho((F,w))$ to be the normalised weighted flag associated to the 1PS $\gamma|_W$. Remark that the action of $\gamma$ on $W$ could be trivial; in this case, $\rho$ is not defined at $(F,w)$. 

\begin{remark}[The map $\rho$ as a retraction] 
By choosing a complement $U$ of $W$ in $V$, there is a natural inclusion of $\Gamma(SL(W))$ in $\Gamma(SL(V))$, which in turn gives an inclusion of building $i\colon \Delta(W)\to \Delta(V)$; the map $\rho$ is the right-inverse of this inclusion.
\end{remark}
We can give the following alternative description of $\rho$, which does not depend the choice of the representative $\gamma$. We let $\rho(F)$ to be the flag defined by $\rho(F)_i W=F_i V \cap W$; we assign weight $w_i$ to $\rho(F)_i$. This definition is ill-posed, and we need to refine it. The first pathology is that $\rho(F)_i W$ is not, in general, a proper sub-space of $W$, if this happens we skip this step of the flag and we relabel the indexes. If all the subspaces $\rho(F)_i W$ are not proper sub-spaces of $W$, then we do not define $\rho$ at $(F,w)$. We can also have repetitions; in symbols, for some $i$, we can have that $\rho(F)_i=\rho(F)_{i+1}$. When this happens, we skip the step $i+1$ of the flag, and we relabel the indexes. To have well-defined flag, we still have to normalise the weight. With this description, we can prove the following lemma.

\begin{lemma}
The retraction $\rho$ defined above is a continuous map for the simplicial topology.
\end{lemma}
\begin{proof}
Since the simplicial topology on $\Delta$ is the direct limit of the topology on the apartments, it is enough to show that $\rho$ restricted to any apartment $A$ is continuous.

Let $A$ be an apartment in $\Delta(V)$ and $\underline{v}$ the corresponding basis. Co-ordinates on $A$ are given by the weights $w$. Let $B$ be an apartment in $\Delta(W)$, and $\underline{u}$ the corresponding basis. The map $\rho \colon U_A\cap \rho^{-1}B\to B$ is given just by the projection onto some of the co-ordinates, i.e. the weigh of $u_i$ in $(\rho(F),\rho(w))$ is just $w_j$ for an appropriate index $j$. This shows that $\rho$ restricted to $A\cap \rho^{-1}B$ is continuous. Since this holds for all apartments $B$ of $\Delta(W)$, we have proven that $\rho$ restricted $A$ is continuous.
\end{proof}

\begin{remark}[Pathologies] Let us stress that $\rho$ does not preserve many geometric features of $\Delta(V)$. To start with, $\rho$ is not open: indeed, already locally on an apartment $A$, we can see that $\rho$ is like a linear projection followed by a linear inclusion, and the latter is not open. This restriction does not preserve neither the simplicial structure nor the apartments. Moreover, $\rho$ does not preserve geodesics. To see this, one can take two flags $F$ and $G$ such that does not exists an apartment which contains $F$, $G$ and $W$, where we see $W$ as a one step flag, so a vertex of $\Delta(V)$. 

\end{remark}

\end{subsection}

\begin{subsection}{Proof of Theorem \ref{thm:cont}}\label{sec:prov}
Let us start off by looking at the Segre morphism
$$
S\colon \P H^0(X,L)^{\vee}\to \P \Sym^k H^0(X,L)^{\vee}
$$
A test configuration $\scX$ embedded $\P H^0(X,L)^{\vee}$ can be re-embedded in a $\G$-equivariant way in $ \P \Sym^k H^0(X,L)^{\vee}$ via $S$. The test configuration $(S(\scX),\O(1))$ is isomorphic to $(\scX,\scL^k)$; in particular, $(S(\scX),\O(1))$ is trivial if and only if $(\scX,\scL^k)$ is trivial. If $\lambda$ is a 1Ps of $SL(H^0(X,L)^{\vee})$ inducing $\scX$, then $S(\gamma)$ induces $(S(\scX),\scL^k)$.

Let now $[\gamma]$ be a point of $\Delta_S$, assume it acts non-trivially on $S(X)$, so that it induces a non-trivial test configuration (we mean non-trivial in the sense of Definition \ref{def:tc}). This test configuration has exponent $k$, because the restriction of $\O(1)$ to $X$ is $L^k$. We define $\rho([\gamma])$ to be the point of $\Delta_k(\Q)$ which represents the test configuration induced by $\gamma$ (we can think at $\Delta_k(\Q)$ as the moduli space of exponent $k$ test configurations, and $\rho$ as a classifying map).

The composition of $\rho$ and $S$ makes sense, because $S(\lambda)$ acts non-trivially on $X$, and it is equal to $\iota$ because of the above discussion. To prove Theorem \ref{thm:cont}, we have to show that the map $\rho$ defined above is the retraction of buildings introduced in Section \ref{sec:retraction}. 

Take $V=\Sym^k H^0(X,L)^{\vee}$ and $W=H^0(X,kL)^{\vee}$. There is a natural inclusion of $W$ in $V$ given by the co-multiplication, let $\rho$ be the associated retraction. The embedding of $X$ in $\P V$ factors trough the embedding of $X$ in $\P W$. Let $(F,w)$ be a weighted flag in $\Delta(V)$, and take a representative $\gamma$ which preserves $W$. Remark that $\rho$ is defined at $(F,w)$ if and only if the action of $\gamma$ on $W$ is not trivial; this is equivalent to ask that the action on $X$ is not trivial, hence that $\gamma$ induces a non-trivial test-configuration. %In particular, if $(F,w)$ is in the image of the Segre map discussed in Section \ref{sec:segre}, then $\rho$ is well-defined at $(F,w)$.

The action of $\gamma$ on $X$ is equal to the action of $\gamma|_W$ on $X$, hence $\rho((F,w))$ represent the test configuration obtained by letting $\gamma$ acting on $X\subset \P V$, as requested. This concludes the proof of Theorem \ref{thm:cont}.

\end{subsection}
\end{section}

\begin{section}{Analogy with classical symmetric spaces}\label{sec:analogy}

In this section we work over the complex numbers. The Tits building $\Delta(V)$ of a vector space $V$ can be introduced as  boundary of the symmetric space $H:=SL(V)/SU(V)$, and this gives also an alternative point of view on apartments, see for example \cite{Borel_Li}. We briefly recall this theory, and suggest an analogy for our Tits building $\Delta_{\infty}$. 
 
The Killing metric on the Lie algebra of $SL(V)$ defines a constant scalar curvature metric with negative curvature on the homogeneous space $H$. Let $o$ be the image of the identity in $H$. A one parameter subgroup $\lambda$ of $SL(V)$ defines a map from $\mathbb{C}^*/\mathbb{S}^1\cong\mathbb{R}^+$ to $H$, the image is a geodesic starting at $o$. One can equivalently define $\Delta(V)$ as the set of all geodesics starting at $o$.  One parameter subgroups are also from this point of view rational points of $\Delta(V)$. It is then possible to define a topology on $\hH:=H \cup \Delta(V)$, which turns $\hH$ into a compact space.

The image of a $d$ dimensional torus of $SL(V)$ in $H$ is a flat subspace, which means that it is isometric to a $d$ dimensional Euclidean space. Maximal tori give maximal flat  subspaces of $H$. One can introduce the notion of rank of $H$ as the dimension of a maximal flat subspace of $H$, and this turns out to be equal to the rank of $SL(V)$. Let $T$ be a maximal torus and $\mathbb{E}_T$ its image in $H$. The boundary of $\mathbb{E}_T$, which can be defined by intersecting the closure of $\mathbb{E}_T$ in $\hH$ with $\Delta(V)$, is the apartment $A_T$ defined in Section \ref{sec:prel}. In a more colloquial language, we can say that apartments are boundaries of maximal flat subspaces.
 
In view of the Yau-Tian-Donaldson conjecture, it is natural to think at $\Delta_{\infty}$ as the boundary of the space $\mathcal{H}$ of K\"{a}hler metrics on $L$. This space has a natural Riemannian metric, as advocated in \cite{Donaldson_Symmetric}. Since $\mathcal{H}$ is infinite dimensional, standard results of Riemannian geometry does not go trough. We can look at $\mathcal{H}$ as a metric space rather than an infinite dimensional manifold, and in this set up $\mathcal{H}$, or rather its completion, is known to be a CAT(0) spaces, see for instance \cite[Theorem 4.11]{ele} and references therein. With the formalism of CAT(0) spaces one can proves many basic result such as the uniqueness of geodesics, see \cite{Bridson}.

This point of view suggests that the notion of apartment for $\Delta_{\infty}$ should be relate to maximally flat subspace in the space of K\"{a}hler potentials. In this set up also the notion of maximally flat subspace is troublesome, we might define them as spaces which are isometric to a Hilbert space. Again, this is discussed in \cite[Section 6]{Donaldson_Symmetric}.

In classical geometric invariant theory, one can define the normalized Hilbert-Mumford weight on $\Delta(V)$ as the slope of the Kempf-Ness functional on $H$, see for instance \cite[Section 5.1]{BHJ2}. The Kempf-Ness functional is convex and Lipschitz, so one can use results from the theory of CAT(0) spaces to prove the existence of maximally destabilising one parameter subgroups, see \cite{Leeb1} and \cite{Leeb2}. One could try a similar approach to study optimally destabilizing test configurations by replacing the Kempf-Ness functional with the Mabuchi or the Ding functional. However, none of these functionals seems to be Lipschitz, so we do not know how to generalize this approach.

\end{section}


\begin{thebibliography}{}

 \bibitem{BHJ}{S. Boucksom, T. Hisamoto, M. Jonsson,Uniform K-stability, Duistermaat-Heckman measures and singularities of pairs, 
Annales de l'Institut Fourier, 2017, 67 (2), pp.743 - 841}
  \bibitem{BHJ2}{S. Boucksom, T. Hisamoto, M. Jonsson,Uniform K-stability and asymptotics of energy functionals in K\"{a}hler geometry, to appear in JEMS }
\bibitem{prep}{S. Boucksom and M. Jonsson. Singular semipositive metrics on line bundles on varieties over trivially valued fields, Arxiv preprint}
% \bibitem{BorelTits} {Borel, Armand and Tits, Jacques, Groupes r\'{e}ductifs, Inst. Hautes \'{E}tudes Sci. Publ. Math., 27, 1965}
\bibitem{Borel_Li}{A. Borel and J. Lie, Compactifications of Symmetric and Locally Symmetric Spaces, Birkhauser, 2006}
 \bibitem{me}{G. Codogni and R. Dervan, Non-reductive automorphism groups, the Loewy filtration and K-stability, Annales de l'institut Fourier Vol. 66 no. 5 (2016) pp. 1895-1921}
 
 
 \bibitem{Bridson}{M. Bridson and A. Haefliger, Metric spaces of non-positive curvature, Grundlehren der Mathematischen Wissenschaften, 319. Springer-Verlag, Berlin, 1999. xxii+643 pp}
 \bibitem{Ruadhai_Uniform}{ R. Dervan, Uniform stability of twisted constant scalar curvature K\"{a}hler metrics, Int. Math. Res. Notices (2016) Vol 15 pp. 4728-4783.}
 \bibitem{DervanRoss}{R. Dervan and J. Ross, K-stability for K\"{a}hler Manifolds, Math. Res. Lett. Vol. 24, No. 3 (2017)}
 \bibitem{Der2}{R. Dervan, Relative K-stability for K\"{a}hler manifolds, to appear in Math. Annalen}
% \bibitem{Ruadhai}{R. Dervan, Relative K-stability for k\"{A}hler manifolds, in preparation}
 \bibitem{ele}{E. Di Nezza, V. Guedj,Geometry and Topology of the space of K\"{a}hler metrics on singular varieties, to appear in Compositio Mathematica}
 \bibitem{Don2}{S. K. Donaldson. Scalar curvature and stability of toric varieties. J. Differential Geom. 62 (2002)}
 \bibitem{Don}{S. K. Donaldson, Lower bounds of Calabi functionals, J. Differential Geom. Volume 70, Number 3 (2005), 453-472}
 \bibitem{Donaldson_Symmetric}{S. Donaldson, Symmetric spaces, K\"{a}hler geometry and Hamiltonian dynamics, Amer. Math. Soc. Transl. Ser. 2, vol. 196; Amer. Math. Soc., Providence RI, 1999, pp. 13-33.}
 \bibitem{MG}{F. Martin and W. Gubler, On Zhang's semipositive metrics, Arxiv preprint}
 \bibitem{Leeb2}{Misha Kapovich, Bernhard Leeb and John Millson, Convex functions on symmetric spaces, side lengths of polygons and the stability inequalities for weighted configurations at infinity,Journal of Differential Geometry 81 (2009), 297-354}
 \bibitem{Leeb1}{B. Kleiner and B. Leeb, Rigidity of quasi-isometries for symmetric spaces and Euclidean buildings, Comptes Rendus de l'Acad\'{e}mie des Sciences - Series I - Mathematics, Volume 324, Issue 6, 1997, Pages 639-643,}
%\bibitem{Bour_Bouc}{S Boucksom, Corps d'Okounkov, S{\'e}minaire Bourbaki, Vol. 2012-13, N.1059, 2012}
%\bibitem{Laz}{R. LAzarsfeld and M. Mustata, Convex bodies associated to linear series}
%\bibitem{GJT}{Y. Guivarc'h, L. Ji, J. C. Taylor, Compactifications of Symmetric Spaces, Progress in Mathematics, Volume 156, 1998, Birkh�user Basel}

\bibitem{GIT}{Mumford, D. and Fogarty, J. and Kirwan, F. , Geometric invariant theory,Ergebnisse der Mathematik und ihrer Grenzgebiete (2), Vol. 34. Springer-Verlag, Berlin, third edition, 1994.}
%\bibitem{Ipossu}{A. Isopossu, K-stability of relative flag varieties, Ph.D Thesis, University of Cambridge, 2015}
\bibitem{Kempf}{G. Kempf, Instability in invariant theory, Annals of Mathematics, Second Series, Vol. 108, No. 2 (Sep., 1978), pp. 299-316}
\bibitem{Laz}{R. Lazarsfeld, Positivity in Algebraic Geometry, Vol 1, Springer 2014}
\bibitem{RossThomas}{J. Ross and R. P. Thomas, A study of the Hilbert-Mumford criterion for the stability of projective varieties, Journal of Algebraic Geometry 16 (2007), 201-255.}
%\bibitem{RossWitt}{Julius Ross and David Witt Nystr\"{o}m, Analytic test configurations and geodesic rays}
\bibitem{Odaka} {Odaka, Yuji, On parametrization, optimization and triviality of test configurations, Proc. Amer. Math. Soc. ,143, 2015}
\bibitem{Rousseau}{G. Rousseau, Immeubles spheriques et theorie des invariants, C. R. Acad. Sci. Paris Ser. A-B 286 (1978)}
%\bibitem{FR}{J. Fine and J. Ross, A note on Positivity of the CM line bundle IMRN, vol. 2006}
\bibitem{Serre}{Serre, Jean-Pierre, Compl\`{e}te r\'{e}ductibilit\'{e}, S{\'e}minaire Bourbaki. Vol. 2003/2004, N. 299, 2005 }
\bibitem{Zak}{Z. Sj\"{o}str\"{o}m Dyrefelt, K-semistability of cscK manifolds with transcendental cohomology class,  To appear in J. Geom. Anal.}
\bibitem{Gabor}{G. Sz\'{e}kelyhidi, Filtrations and test-configurations, Math. Ann. (2015) 362: 451}
\bibitem{ICM}{G. Sz\'{e}kelyhidi, Extremal K\"{a}hler metrics, Proceedings of the ICM, 2014}
\bibitem{witt}{D. Witt-Nystrom,Test configurations and Okounkov bodies, Compositio Mathematica 148, 6 (2012), pp. 1736-1756.}
\end{thebibliography}
\end{document}